\documentclass[12pt,a4paper]{amsart}

 \usepackage{todonotes}
 \usepackage{csquotes}

\newenvironment{proof:mainthm1}{\vspace{0.2cm}\paragraph{\bf\textit{Proof  of Theorem \ref{thm:main1}:}}}{\hfill$\blacksquare$ }
\newenvironment{proof:mainthm2}{\vspace{0.2cm}\paragraph{\bf\textit{Proof  of Theorem \ref{thm:main2}:}}}{\hfill$\blacksquare$ }
\newenvironment{proof:corollary}{\vspace{0.2cm}\paragraph{\bf\textit{Proof  of Corollary \ref{cor:corollary}:}}}{\hfill$\blacksquare$ }
\newenvironment{proof:corollary2}{\vspace{0.2cm}\paragraph{\bf\textit{Proof  of Corollary \ref{cor:examples}:}}}{\hfill$\blacksquare$ }

\usepackage{amsfonts,amssymb,enumerate,bm,hhline,makecell,array,mathtools}
\usepackage{mathrsfs}
\usepackage{comment}
\usepackage[normalem]{ulem}
\usepackage[colorlinks=true,citecolor=blue!70!black, urlcolor=blue!70!black, pagebackref, linkcolor=blue!70!black]{hyperref}
\usepackage{tikz-cd}
\usepackage{amscd}
\usepackage[shortlabels]{enumitem}
\setlist[itemize]{noitemsep,nolistsep}
\usepackage{tensor}
\usepackage{tikz}
\usepackage{setspace}
\usetikzlibrary{matrix,arrows}
\usepackage{extarrows}
\usepackage[toc,page]{appendix}
\usepackage[all,ps,cmtip,rotate]{xy} 
\usepackage{color}
\usepackage{hyperref}
\hypersetup{
    colorlinks,
    citecolor = green!80!black,
    linkcolor=red,
    urlcolor=magenta
}
\usepackage{lmodern}
\usepackage[all]{xy}
\usepackage{comment}

\usepackage{enumitem}
\setlist[itemize]{noitemsep,nolistsep}
\setlist[enumerate]{noitemsep,nolistsep}
\usepackage{colortbl} 

\allowdisplaybreaks

\def\Z{{\bf Z}}

\def\C{{\bf C}}

\def\R{{\bf R}}
\def\Q{{\bf Q}}
\def\P{{\bf P}}

\def\hk{hyper-K\"ahler}

\def\phi{\varphi}

\def\cF{\mathcal{F}}
\def\cG{\mathcal{G}}

\def\cL{\mathcal{L}}

\def\cP{\mathcal{P}}

\def\cS{\mathcal{S}}

\def\cX{\mathscr{X}}

\def\lra{\longrightarrow}
\def\llra{\hbox to 10mm{\rightarrowfill}}
\def\lllra{\hbox to 15mm{\rightarrowfill}}

\def\llla{\hbox to 10mm{\leftarrowfill}}
\def\lllla{\hbox to 15mm{\leftarrowfill}}

\DeclareMathOperator{\Aut}{Aut}
\DeclareMathOperator{\Bir}{Bir}

\DeclareMathOperator{\Hdg}{Hdg}

\DeclareMathOperator{\id}{id}

\DeclareMathOperator{\Int}{int}

\DeclareMathOperator{\Mon}{Mon}
\DeclareMathOperator{\Mov}{Mov}

\DeclareMathOperator{\rank}{rank}

\def\llra{\hbox to 10mm{\rightarrowfill}}
\def\lllra{\hbox to 15mm{\rightarrowfill}}

\def\subset{\subseteq}

\newtheorem{lem}{Lemma}[section]

\newtheorem{thm}[lem]{Theorem}
\newtheorem{cor}[lem]{Corollary}
\newtheorem{prop}[lem]{Proposition}

\theoremstyle{definition}
\newtheorem{defin}[lem]{Definition}
\newtheorem{rem}[lem]{Remark}

\newtheorem{ex}[lem]{Example}

\theoremstyle{remark}
\newtheorem*{remark*}{Remark}
\newtheorem*{note*}{Note}

\def\Lkkk[#1]{{\Lambda_{\KKK^{[#1]}}}}
\def\kkk[#1]{{\KKK^{[#1]}}}
\DeclareMathOperator{\KKK}{{K3}}

\def\sss[#1]{{S^{[#1]}}}

\def\setminus{\smallsetminus}

\definecolor{orange}{rgb}{1,0.55,0}

\definecolor{pink}{rgb}{1,0.65,0.79}

\begin{document}
\title{Birational automorphism groups in families of hyper-Kähler manifolds}

\author[]{}
\address{}
 \email{}

    \author[F.\ A.\ Denisi]{Francesco Antonio Denisi}
    \address{Francesco Antonio Denisi, Fachrichtung Mathematik, Campus, Gebäude E2 4, Universität des Saarlandes, 66123 Saarbrücken, Deutschland}
    \email{denisi@math.uni-sb.de}

    \author[C.\ Onorati]{Claudio Onorati}
    \address{Claudio Onorati, Dipartimento di Matematica, Universit\`a degli Studi di Bologna, Piazza di Porta San Donato 5, 40126 Bologna, Italia.}
    \email{claudio.onorati@unibo.it}

    \author[F.\ Rizzo]{Francesca Rizzo}
    \address{Francesca Rizzo, Université Paris Cité and Sorbonne Université, CNRS, IMJ-PRG, F-75013 Paris, France}
    \email{francesca.rizzo@imj-prg.fr}

    \author[S.\ Viktorova]{Sasha Viktorova}
    \address{Sasha Viktorova, Mathematics Division, National Center for Theoretical Sciences, National Taiwan University, Taipei City 106, Taiwan}
    \email{sasha.viktorova@ncts.ntu.edu.tw}


\thanks{}

\begin{abstract}
We study the behavior of birational automorphism groups in families of projective hyper-Kähler manifolds.
\end{abstract}

\maketitle

\setcounter{tocdepth}{1}

\tableofcontents

 \section{Introduction}
Given a projective variety \(Y\), the groups \(\Aut(Y)\) and \(\Bir(Y)\) play an important role in understanding the geometry of \(Y\). For example, very recently it has been proven that $\mathrm{Bir}(Y)$ determines whether a projective variety is rational or ruled \cite{RUvS25}. Moreover, they are central objects in some important conjectures in algebraic geometry. Among these, one of the most significant ones is the Kawamata--Morrison cone conjecture, which roughly states that the nef (respectively, movable) cone of a Calabi--Yau manifold \(Y\) (in the broader sense), defined as the cone in the Néron--Severi space spanned by nef (respectively, movable) divisor classes, is potentially rational polyhedral. If this is not the case, the obstruction is given by the group \(\Aut(Y)\) (respectively, \(\Bir(Y)\)), which in such situations is infinite.

Our interest in the groups \(\Aut(Y)\) and \(\Bir(Y)\) stems from the fact that when a projective variety \(Y\) is \(K\)-trivial (that is, when the canonical divisor \(K_Y\) is numerically trivial), these groups provide information on whether \(Y\) is a \emph{Mori dream space}. Roughly speaking, a Mori dream space is a normal and \(\Q\)-factorial projective variety $X$ with vanishing irregularity behaving nicely under the Minimal Model Program. Algebraically, being a Mori dream space is characterized as follows. Let $\Gamma$ be a finitely generated group of Weil divisors mapping surjectively onto the class group. Then $X$ is a Mori dream space if and only if the algebra
$$
\bigoplus_{D\in \Gamma}H^0(X,\mathscr{O}_X(D))
$$
is finitely generated. To give a concrete example, Fano varieties (that is, smooth projective varieties with ample anticanonical divisor) are Mori dream spaces (see \cite{BCHM10}), and a K3 surface \(S\) is a Mori dream space if and only if its automorphism group \(\Aut(S)\) is finite (see \cite{AHL:MoriDreamK3}). Conjecturally, a projective {\hk} manifold $X$ is a Mori dream space if and only if the group $\mathrm{Bir}(X)$ is finite (see \cite[Corollary 3.10]{Den22}).  Recall that a \hk\ manifold is a simply connected compact Kähler manifold \(X\) such that $H^0(X,\Omega_X^2)=\mathbf{C}\sigma$, where \(\sigma\) is a holomorphic symplectic form.

A natural problem is to study the behavior of the groups \(\Aut(Y)\) and \(\Bir(Y)\) in families. In this paper, a \emph{family} is a surjective, proper and flat morphism $\varphi\colon \mathscr{X}\to S$ of complex analytic spaces, with $S$ connected. We say that the family \(\varphi\) is polarized (or projective) if there exists a \(\varphi\)-ample line bundle \(\cL\).  Oguiso proves in \cite{Oguiso03} that for a nontrivial family $\varphi\colon \mathscr{X} \to \Delta$ of projective K3 surfaces over a disk $\Delta$, there exists a dense subset of \(\Delta\) over which the fibers have infinite automorphism group. In particular, being a Mori dream space is not an open condition in families of K3 surfaces. This phenomenon does not occur, for example, with Fano varieties:
for a smooth family $\mathscr{X}\to (\Delta,0)$ whose central fiber $\mathscr{X}_0$ is Fano, the fiber $\mathscr{X}_s$ is also Fano (and hence a Mori dream space) for $s$ close enough to $0$.

The purpose of this note is to study the behavior of birational automorphism groups in families of projective \hk\ manifolds, thereby generalizing Oguiso's results on K3 surfaces. More precisely, we focus on families over a disk. Then, let \(\varphi\colon \mathscr{X} \to \Delta\) be a nontrivial family of \hk\ manifolds over the unit disk, endowed with the Euclidean topology. By \cite[Theorem~1.1]{Oguiso03}, the Picard rank of the fibers attains its minimum on a subset \(\cG \subset \Delta\), whose complement \(\cS\subset \Delta\) is dense and countable.

We first study the group of birational automorphisms of a very general fiber $\cX_t$, with $t\in \cG$.
\begin{thm}\label{thm:main1}
    Let $\varphi\colon\mathscr{X}\to \Delta$ be a nontrivial polarized family of \hk\ manifolds, 
    over a small disk $\Delta$.  
    There exists a (possibly empty) finite subset $\cF\subset \cS$, a group $G^0$ and a constant $N \coloneqq N(\varphi)$ such that
    $$
        G^0<{\Mon}^2_{\Bir}(\mathscr{X}_t) \quad  \text{for all}\ t\in \Delta\setminus \cF,
    $$
    and
    $$
    [{\Mon}^2_{\Bir}(\mathscr{X}_t) : G^0]\le N \quad \text{for all}\ t\in\cG.
    $$
    In particular, the map 
    \begin{align*}
         \Bir \colon \Delta &\lra \{\text{groups}\}/\equiv\\
         t&\mapsto \Bir(\mathscr{X}_t),
    \end{align*}
    is \enquote{upper-semicontinuous} with respect to the co-finite topology on $\Delta\setminus \cF$, in the sense that the index of $G^0$ in ${\Mon}^2_{\Bir}(\mathscr{X}_t)$ can only become infinite at special points in $\mathcal{S}\setminus\mathcal{F}$.
\end{thm}
In particular, Theorem~\ref{thm:main1} shows that if $\Bir(\cX_{t_0})$ is finite for some $t_0 \in \cG$, then $\Bir(\cX_t)$ is finite for all $t \in \cG$. By contrast, special fibers may exhibit different behavior, acquiring an infinite group of birational automorphisms.  We show that, if the fibers of the polarized family belong to one of the known deformation classes (K$3^{[n]}$, $\mathrm{Kum}^n$, OG$6$, or OG${10}$), this phenomenon always occurs on a dense, countable subset. Note that it may happen that all the members of the family have an infinite group of birational automorphisms (see Example \ref{ex:fanooflines}.(1)).

\begin{thm}\label{thm:main2}
 Let $\varphi\colon\mathscr{X}\to \Delta$ be a nontrivial polarized family of \hk\ manifolds, 
    over a small disk $\Delta$, whose fibers belong to one of the known deformation classes of hyper-Kähler manifolds. Then there is a subset $\mathcal{D}\subset \mathcal{S}$ such that $\mathcal{D}$ is dense in $\Delta$ and such that 
    \[
    |\Bir(\mathscr{X}_t)|=\infty, \text{ for all } t \in \mathcal{D}
    \]
\end{thm}
Theorem \ref{thm:main2} tells us that, as one may expect, for all known deformation classes of {\hk} manifolds, the property of being a Mori dream space is not open in families. An immediate consequence of Theorem \ref{thm:main2} is the following.
\begin{cor}\label{cor:corollary}
Let $\varphi \colon \cX \to \Delta$ be a nontrivial family of projective hyper-Kähler manifolds belonging to one of the known deformation types over a disk $\Delta$. 
Then there exists a dense subset $\mathcal{D} \subset \Delta$ such that 
\[
    |\Bir(\cX_t)| = \infty
    \quad \text{for all } t \in \mathcal{D}.
\]
\end{cor}

In particular, Corollary \ref{cor:corollary} allows us to drop the projectivity assumption in Theorem \ref{thm:main2}.

We conclude this introduction with two examples applying Theorem \ref{thm:main1} and Theorem \ref{thm:main2}.

\begin{ex}\label{ex:fanooflines} The first example provides families of hyper-Kähler fourfolds whose members all have infinite birational automorphism groups. The second provides families of hyper-Kähler fourfolds for which the very general member has Picard rank $1$ (and hence a finite birational automorphism group), whereas the special members have an infinite birational automorphism group and form a dense subset of the base of the family.
\begin{enumerate}
\item Let \(\pi \colon \mathscr{X} \to (\Delta,0)\) be a smooth projective family of cubic fourfolds over a small disk such that the central fiber lies on more than one Hassett divisor, while the very general fiber lies on the Hassett divisor \(\mathcal{C}_{12}\). Consider the relative Fano variety of lines
\[
F_1(\pi)\colon F_1(\mathscr{X}/\Delta) \to \Delta,
\]
which is a nontrivial family under our assumptions. The very general fiber of \(F_1(\pi)\) is a {\hk} fourfold with infinite birational automorphism group, by \cite[Theorem~31]{HT10}, and has Picard number equal to \(2\). By Theorem~\ref{thm:main1}, the special fibers of \(F_1(\pi)\) with infinite birational automorphism group form a dense subset of \(\Delta\). Hence, after removing finitely many points from \(\Delta\) (namely, those lying in \(\mathcal{F}\)), all remaining fibers have infinite birational automorphism group.

An analog example can be constructed using the divisor \(\mathcal{C}_{20}\) and \cite[Lemma~4.3]{BFM24}, which asserts that the Fano variety of lines of a very general element of \(\mathcal{C}_{20}\) has an infinite birational automorphism group.

\item Let \(\pi \colon \mathscr{X} \to (\Delta,0)\) be a smooth projective family of cubic fourfolds over a small disk such that the central fiber lies on some Hassett divisor, while the very general fiber $\mathscr{X}_s$ does not lie on any Hassett divisor, that is,
\[
H^4(\mathscr{X}_s,\mathbf{Z}) \cap H^{2,2}(\mathscr{X}_s,\mathbf{C}) = \mathbf{Z} h^2 .
\]
Consider the associated relative Fano variety of lines
\[
F_1(\pi)\colon F_1(\mathscr{X}/\Delta) \to \Delta .
\]
By assumption, the family \(F_1(\pi)\) is nontrivial. The very general fiber of \(F_1(\pi)\) is a {\hk} fourfold with finite birational automorphism group, as its Picard number is equal to \(1\). By Theorem~\ref{thm:main2}, the special fibers of \(F_1(\pi)\) with infinite birational automorphism group form a dense subset of \(\Delta\).
\end{enumerate}
\end{ex}

\subsection*{Acknowledgments} 
We thank Simon Brandhorst and Keiji Oguiso for useful discussions, and Pietro Beri for bringing \cite[Lemma 3.3]{BMP25} to our attention, which led us to the proof of Proposition \ref{prop:finmanylattices}. Denisi and Rizzo were supported by the European Research Council (ERC) under European Union's Horizon 2020 research and innovation programme (ERC-2020-SyG-854361-HyperK). Denisi was also supported by the Deutsche Forschungsgemeinschaft (DFG, German Research Foundation) – Project ID 286237555 (TRR 195). Onorati was supported by the European Union - NextGenerationEU under the National Recovery and Resilience Plan (PNRR) - Mission 4 Education and research - Component 2 From research to business - Investment 1.1 Notice Prin 2022 - DD N. 104 del 2/2/2022, from title “Symplectic varieties: their interplay with Fano manifolds and derived categories”, proposal code2022PEKYBJ – CUP J53D23003840006.

\section{Preliminaries}
In this section, we collect the main definitions, tools, and results needed for the rest of the paper. For a general introduction to {\hk} manifolds, we refer the reader to \cite{GHJ03}. Given a {\hk} manifold $X$, we recall that, thanks to the work \cite{Beauville83a} of Beauville, there exists a non-degenerate quadratic form $q_X$ on $H^2(X,\mathbf{C})$, known as the Beauville-Bogomolov-Fujiki form (BBF form in what follows). 
\subsection{Lattice theory and {\hk} manifolds}

A lattice is a pair $(L,b)$, where $L$ is a finitely generated, free abelian group, and $b$ is a non-degenerate, integer-valued, symmetric bilinear form $b \colon L \times L \to \mathbf{Z}$. If no confusion arises, we will write $L$ and set $(x,y):=b(x,y)$, for every $x,y \in L$. We say that  $L$ is even if $b(x,x)$ is even for any $x \in L$. The signature $\mathrm{sign}(b)$ of $b$ is the signature of the natural extension $b_{\mathbf{R}}$ of $b$ to $L \otimes_{\mathbf{Z}} \mathbf{R}$. The divisibility $\mathrm{div}_L(x)$ of an element $x \in L$ is defined as the positive generator of the ideal $b(x,L)$ in $\mathbf{Z}$. A \emph{rescaling} of a lattice $L$ is $L(k)$ for some integer $k$, that is, the matrix of the bilinear form of $L$ is multiplied by $k$.

As $b$ is non-degenerate, one has an injective group homomorphism $L \hookrightarrow L^{\vee}:=\mathrm{Hom}_{\mathbf{Z}}(L,\mathbf{Z})$, by sending an element $x$ of $L$ to the element $b(x,-)$ of $L^{\vee}$. Note that we have the following identification
\[
L^{\vee}= \{x\in L\otimes_{\mathbf{Z}} \mathbf{Q} \; | \; b(x,y) \in \mathbf{Z}, \text{ for any } y \in L\}\subset L \otimes_{\mathbf{Z}} \mathbf{Q}.
\]

The discriminant group of $L$ is the finite group $D_L:=L^{\vee}/L$. If $D_L$ is trivial we say that $L$ is unimodular. We say that an element $x \in L$ is primitive if we cannot write $x=ax'$, where $a\neq 1$, $a \in \mathbf{Z}_{>0}$, and $x' \in L$. 

An isometry of $L$ is an automorphism of it (as an abelian group) preserving $b$.  The group of isometries of $L$ is denoted by $\text{O}(L)$. Suppose now that $L$ has signature $\mathrm{sign}(b)=(3,b_2-3)$, where $b_2 \in \mathbf{N}$, and $b_2>3$. Define the cone
\[
C_L:=\{v \in L\otimes_{\mathbf{Z}} \mathbf{R} \;|\; b(v)>0\}.
\] 
In \cite[Lemma 4.1]{Markman11} it has been proven that $C_L$ has the homotopy type of $S^2$, hence $H^2(C_L,\mathbf{Z})\cong \mathbf{Z}$. Any isometry in $\text{O}(L)$ induces a homeomorphism of $C_L$, which in turn induces an automorphism of $H^2(C_L,\mathbf{Z})\cong \mathbf{Z}$. This automorphism can act as $1$ or $-1$. We define $\text{O}^+(L)$ as the subgroup of the isometries of $L$ acting trivially on $H^2(C_L,\mathbf{Z})$.

Given a {\hk} manifold $X$, one has the lattice $(H^2(X,\mathbf{Z}),q_X)$, where (with some abuse of notation) by $q_X$ we mean the integral valued, symmetric bilinear form induced by the restriction of the BBF form to $H^2(X,\mathbf{Z})$. Note that, for any of the known {\hk} manifolds, the lattice $H^2(X,\mathbf{Z})$ is even and contains three copies of the hyperbolic lattice $U$ (see, for example, \cite{Rapagnetta08}). The signature of $q_X$ on $H^2(X,\mathbf{Z})$ is $(3, b_2(X)-3)$, where $b_2(X)$ is the second Betti number of $X$. We will denote the discriminant group $D_{H^2(X,\mathbf{Z})}$ by $D_X$. Also $(\mathrm{Pic}(X),q_X)$ is a lattice, and in this case, the signature of $q_X$ is $(1,\rho(X)-1)$, where $\rho(X)$ is the Picard number of $X$. We refer the reader to \cite[Corollary 23.11]{GHJ03} for proof of these facts.

The following lemma will be useful for the proof of Theorem~\ref{thm:main2}.
\begin{lem}
    Let $M = U^{\oplus 3}\oplus T$ be an \emph{even} lattice. For every primitive vector $\ell\in M$ and every vector $a\in M$, the lattice $\langle \ell, a\rangle^\perp$ contains a hyperbolic plane.
\end{lem}
\begin{proof}
    For every $i=1,2,3$, denote by $e_i, f_i$ a basis of each copy of $U$ inside $M$.
    
    Write $\ell$ as $am + t$, where $m$ is a primitive vector in $U^{\oplus 3}$. Since $U^{\oplus 3}$ is unimodular, by Eichler's Lemma, there exists an isomoetry $f_1$ of $U^{\oplus 3}$ that sends $m$ to some vector $u_1= e_1 - bf_1$. Since $U^{\oplus 3} = T^\perp$ is unimodular, and the lattice $T$ is primitive in $M$, by \cite[Corollary 1.5.2]{Nikulin80} (see \cite[Section~1.1]{rizzo2024} for more details), the isometry $f_1$ extends to some isometry $\widetilde f_1$ of $M$, which is the identity on $T$. Hence, we can suppose that the vector $\ell$ is of the form $a(e_1-bf_1)+t$. Therefore, the lattice $\ell^\perp$ decomposes as $U^{\oplus 2} \oplus T_1$. Note that the vector $b \coloneqq (\ell^2)a - (a\cdot \ell) \ell$ lies in $\ell^\perp$ and that we have equality of lattices $\langle \ell, a\rangle^\perp=\langle \ell, b\rangle^\perp$. Hence it is enough to show that the $\beta^\perp\subset \ell^\perp$ contains a copy of $U$, and this is done using the same reasoning as above.
\end{proof}
\subsection{Monodromy operators and cones in the Néron--Severi space}
Let $\pi \colon \mathscr{X}\to S$ be a family of {\hk} manifolds over a reduced base $S$, whose central fiber is a fixed {\hk} manifold $X$. Let $R^k\pi_{*}\mathbf{Z}$ be the $k$-th higher direct image of $\mathbf{Z}$ . The space $S$ is locally contractible, and the family $\pi$ is topologically locally trivial (see for example \cite[Theorem 14.5]{GHJ03}). The sheaf $R^k\pi_{*}\mathbf{Z}$ is the sheafification of the presheaf
\[
U \mapsto H^k(\pi^{-1}(U),\mathbf{Z}), \; \text{ for any open subset } U\subset S,
\]   
which is a constant presheaf, by the local contractibility of $S$ and the local triviality of $\pi$. Then $R^k\pi_{*}\mathbf{Z}$ is a locally constant sheaf, i.e. a local system, for every $k \in \mathbf{N}$. Now, let $\gamma \colon [0,1] \to S$ be a continuous path. Then $\gamma^{-1}\left(R^k\pi_{*}\mathbf{Z}\right)$ is a constant sheaf for every $k$.

\begin{defin}
\begin{enumerate}
\item Set $\pi^{-1}(\gamma(0))=X$ and $\pi^{-1}(\gamma(1))=X'$. The parallel transport operator $T^k_{\gamma, \pi}$ associated with the path $\gamma$ and $\pi$ is the isomorphism $$T^k_{\gamma,\pi}\colon H^k(X,\mathbf{Z})\to H^k(X',\mathbf{Z})$$ between the stalks at 0 and 1 of the sheaf $R^k\pi_{*}\mathbf{Z}$, induced by the trivialization of $\gamma^{-1}\left(R^k\pi_{*}\mathbf{Z}\right)$. The isomorphism $T^k_{\gamma,\pi}$ is well defined on the fixed endpoints homotopy class of $\gamma$.
\item If $\gamma$ is a loop, we obtain an automorphism $T^k_{\gamma,\pi}\colon H^k(X,\mathbf{Z})\to H^k(X,\mathbf{Z})$. In this case $T^k_{\gamma,\pi}$ is called a monodromy operator.
\item The $k$-th group of monodromy operators on $H^k(X,\mathbf{Z})$ induced by $\pi$ is
\[
\mathrm{Mon}^k(X)_{\mathbf{\pi}}:=\{T^k_{\gamma,\pi} \; | \; \gamma(0)=\gamma(1) \}
\]
\end{enumerate}
\end{defin}
With the definition above we can define the monodromy groups of a {\hk} manifold $X$.
\begin{defin}
The $k$-th monodromy group $\mathrm{Mon}^k(X)$ of a {\hk} manifold $X$ is defined as the subgroup of $\mathrm{Aut}_{\mathbf{Z}}(H^k(X,\mathbf{Z}))$ generated by the subgroups of the type  $\mathrm{Mon}^k(X)_{\mathbf{\pi}}$, where $\pi \colon \mathscr{X} \to S$ is a family of {\hk} manifolds over a reduced analytic base. The elements of $\mathrm{Mon}^k(X)$ are also called monodromy operators.
\end{defin}

We are interested in the group $\mathrm{Mon}^2(X)$. We recall that $\mathrm{Mon}^2(X)$ is contained in $\text{O}^+(H^2(X,\mathbf{Z}))$. We will mostly use three subgroups of $\mathrm{Mon}^2(X)$. The first is $\mathrm{Mon}^2_{\mathrm{Bir}}(X)$, defined as the elements of $\mathrm{Mon}^2(X)$ induced by the birational self maps of $X$. Indeed, any birational self-map induces a monodromy operator on $H^2(X,\mathbf{Z})$, by a result of Huybrechts (cf. \cite[Corollary 2.7]{Huybrechts03}). The second is $\mathrm{Mon}_{\mathrm{Hdg}}^2(X)$, namely the subgroup of the elements of $\mathrm{Mon}^2(X)$ which are also Hodge isometries. The last one is $W_{\mathrm{Exc}}$, that is, the group of isometries of $H^2(X,\mathbf{Z})$ generated by reflections by prime exceptional divisor classes. Indeed, Markman proved that any reflection $R_E$ associated with a prime exceptional divisor $E$ is an element of $\mathrm{Mon}_\mathrm{Hdg}^2(X)$ (\cite[Corollary 3.6, item (1)]{Markman13}).

Let $\mathrm{Def}(X)$ be the Kuranishi deformation space of a \hk\ manifold $X$ and let $\pi \colon \mathscr{X}\to \mathrm{Def}(X)$ be the universal family. Also, let $0$ be a distinguished point of $\mathrm{Def}(X)$ such that $\mathscr{X}_0 \cong X$. Set $\Lambda=H^2(X,\mathbf{Z})$. By the local Torelli Theorem, $\mathrm{Def}(X)$ embeds holomorphically into the period domain
\[
\Omega_{\Lambda}:=\{p \in \mathbf{P}(\Lambda \otimes_{\mathbf{Z}} \mathbf{C}) \mid q_X(p)=0,\quad \text{and}\quad q_X(p+\overline{p})>0\}
\]
as an analytic open subset, via the local period map
\[
\mathcal{P}\colon \mathrm{Def}(X)\to \Omega_{\Lambda}, \; t \mapsto \left[H^{2,0}(\mathscr{X}_t)\right].
\]
Let $L$ be a holomorphic line bundle on $X$. The Kuranishi deformation space of the couple $(X,L)$ is defined as $\mathrm{Def}(X,L):= \mathcal{P}^{-1}(c_1(L)^{\perp})$, where $c_1(L)^{\perp}$ is a hyperplane in $\mathbf{P}(\Lambda \otimes_{\mathbf{Z}} \mathbf{C})$. The space $\mathrm{Def}(X,L)$ is the part of $\mathrm{Def}(X)$ where $c_1(L)$ stays algebraic. Up to shrinking $\mathrm{Def}(X)$ around $0$, we can assume that $\mathrm{Def}(X)$ and $\mathrm{Def}(X,L)$ are contractible. Let $s$ be the flat section (with respect to the Gauss-Manin connection) of $R^2\pi_{*}\mathbf{Z}$ through $c_1(L)$ and $s_t \in H^{1,1}(\mathscr{X}_t,\mathbf{Z})$ its value at $t \in \mathrm{Def}(X,L)$.

\begin{defin}\label{stablyexc}
A line bundle $L$ on a {\hk} manifold $X$ is \textit{stably exceptional} if there exists a closed analytic subset $Z \subset \mathrm{Def}(X,L)$, of positive codimension, such that the linear system $|L_t|$ consists of a prime exceptional divisor for every $t \in \mathrm{Def}(X,L) \setminus Z$. An integral divisor $D$ is stably exceptional if $\mathscr{O}_X(D)$ is.
\end{defin}

We say that a prime divisor $E$ on a {\hk} manifold $X$ is exceptional if $q_X(E)<0$. Prime exceptional divisors are stably exceptional, by \cite[Proposition 5.2]{Markman10}. By making use of parallel transport operators, we can define \emph{stably exceptional classes}.

\begin{defin}\label{stablyexc1}
Let $X$ be a projective {\hk} manifold. A primitive, integral divisor class $\alpha \in N^1(X)$ (the Néron--Severi group of $X$) is stably exceptional if $q_X(\alpha,A)>0$, for some ample divisor $A$, and there exist a projective {\hk} manifold $X'$ and a parallel transport operator 
\[
f \colon H^2(X,\mathbf{Z})\to H^2(X',\mathbf{Z}),
\] 
such that $kf(\alpha)$ is represented by a prime exceptional divisor on $X'$, for some integer $k$.
\end{defin}

For example, a line bundle of self-intersection $-2$, intersecting positively any ample line bundle on a (smooth) projective K3 surface is stably exceptional, and its class in the Néron--Severi space is stably exceptional. Note that any stably exceptional class (line bundle) is effective, by the semicontinuity theorem.

The following lemma will play a crucial role in the proof of Theorem \autoref{thm:main2}.

\begin{prop}\label{prop:finmanylattices}
 Let $X$ be a projective hyper-Kähler manifold with $\rho(X)\geq 3$ and $\mathrm{Bir}(X)$ finite. Then there are finitely many lattices representing $\mathrm{Pic}(X)$.
\end{prop}
\begin{proof}
 By \cite[Corollary 1.8]{Den22} the index of $W_{\mathrm{Exc}}$ in O$(\mathrm{Pic}(X))$ is finite. Then by \cite[Theorem 5.2.1]{Nikulin82}, up to a rescaling of $\mathrm{Pic}(X)$ as a lattice, there are finitely many lattices $\Lambda_0$ representing $\mathrm{Pic}(X)$. But by \cite[Proposition 4.8]{KMPP19}, once we fix the deformation type of $X$, we know that the prime exceptional divisors have square bounded from below. Thus, for any possible $\Lambda$, there are finitely many rescalings that could represent $\mathrm{Pic}(X)$, because $\Lambda_0$ represents at least one prime exceptional divisor. This concludes the proof.
\end{proof}

Let $Y$ be a projective variety. Recall that the pseudo-effective cone $\overline{\mathrm{Eff}}(Y)$ of $Y$ is the closure of the effective cone $\mathrm{Eff}(Y)$ (or, equivalently, of the big cone $\mathrm{Big}(Y)$) in the Néron--Severi space $N^1(Y)_{\mathbf{R}}$, which is defined as $N^1(Y)_{\mathbf{R}}:=N^1(Y) \otimes \mathbf{R}$, where $N^1(Y)$ is the Néron--Severi group of $Y$. We notice that the big cone is the interior of $\overline{\mathrm{Eff}}(Y)$. For an excellent account of cones in the Néron--Severi space on arbitrary complex projective varieties, we refer the reader to \cite{Lazarsfeld04}.

A line bundle on $Y$ is movable if $|mL|$ has no divisorial components in its base locus, for $m$ sufficiently large and divisible. A Cartier divisor $D$ is movable if $\mathscr{O}_Y(D)$ is. The movable cone of $Y$ is the cone $\mathrm{Mov}(Y)$ in $N^1(Y)_{\mathbf{R}}$ generated by the movable Cartier divisor classes. 

We will use the following characterization for the interior of the movable cone of a projective {\hk} manifold, for which we refer the reader to \cite[Remark 9]{HT09} and \cite[Subsection 2.2]{Den23}.
\begin{lem}\label{lem:charMov}
    Let $X$ be a projective \hk \ manifold. A class $h\in N^1(X)$ lies in the interior of the movable cone if and only if $q_X(h)>0$ and
    $
            q_X(h, E)>0
    $
     for every prime exceptional divisor $E$ on $X$.
\end{lem} 

Note that a class \(h\) lying in the interior of the movable cone satisfies $q_X(h,E) > 0$ for every stably exceptional divisor. Indeed, any stably exceptional divisor \(E\) is effective, and since \(q_X(E) < 0\), there exists a prime exceptional divisor \(E'\) such that $E = E' + D$,
where \(D\) is an effective divisor. It then follows from the lemma above that \(q_X(h,E') > 0\). Moreover, since \(h\) lies in the movable cone, we have \(q_X(h,D) \geq 0\). Therefore,
\[
q_X(h,E) = q_X(h,E') + q_X(h,D) > 0,
\]
as claimed.

If $X$ belongs to one of the known deformation types (i.e.,\ $\operatorname{K3}^{[n]}$, $\operatorname{Kum}^n$, OG6, or OG10), then we can give an abstract description of the (interior of the) movable cone. First, let us denote by $\eta$ the deformation class, and let $\Lambda_{\eta}$ be the abstract lattice such that $\Lambda_{\eta} \cong H^2(X,\mathbf{Z})$ for any $X$ of class $\eta$.

By the work of Markman~\cite{Markman13} (see also~\cite{BM:MMP}) in the $\operatorname{K3}^{[n]}$ case, Yoshioka~\cite{Yoshioka:KumBridg} in the $\operatorname{Kum}^n$ case, Mongardi--Rapagnetta~\cite{MR:OG6} in the OG6 case, and Mongardi--Onorati~\cite{MO:OG10} in the OG10 case, there exists a subset $\mathcal{S}_\eta \subset \Lambda_{\eta}$ with the property that, for any marked pair $(X,\theta)$ of class $\eta$, the set of stably exceptional classes is
\[
\mathcal{S}_X :=
\left\{
x \in \theta^{-1}(\mathcal{S}_\eta) \cap N^1(X)
\;\middle|\;
q_X(x,\kappa) > 0 \text{ for some K\"ahler class } \kappa.
\right\}
\]

    The condition $q_X(x,\kappa) > 0$ ensures that such classes lie in the effective cone of $X$.

    \begin{rem}\label{rmk:S tau}
    If $\eta$ is the deformation class $\operatorname{K3}^{[n]}$ with $n-1$ a prime power, or $\operatorname{Kum}^n$ with $n+1$ a prime power, or OG6, or OG10, then $\mathcal{S}_\eta$ is intrinsically determined in $\Lambda_{\eta}$. In this case, it is also uniquely determined, up to an isometry of $\Lambda_{\eta}$. In the other cases, the set $\mathcal{S}_\eta$ is determined once a primitive embedding of $\Lambda_{\eta}$ inside the Mukai lattice is chosen. Since there are only finitely many such primitive embeddings (up to isometries of the Mukai lattice), we have finitely many choices for $\mathcal{S}_\eta$, again up to an isometry of $\Lambda_{\eta}$.
    \end{rem}

    We can use this numerical characterization of stably exceptional divisors to prove that a given Néron--Severi lattice supports at most finitely many movable cones. Let us make this precise by setting the following definition.

    \begin{defin}
    Let $\Lambda_0$ be a lattice of signature $(1,n)$.
    \begin{enumerate}
    \item A subset $F\subset\Lambda_0 \otimes \R$ is an \emph{abstract $\eta$-movable cone} if there exist a primitive embedding $\iota \colon\Lambda_0\to\Lambda_{\eta}$, a projective {\hk} manifold $X$, and a marking $\theta \colon H^2(X,\Z) \to \Lambda_{\eta}$, such that $N^1(X)_{\R}=\theta_{\R}^{-1}(
\iota(\Lambda_0\otimes \R))$ and $\mathrm{int}(\mathrm{Mov}(X))=\theta_{\R}^{-1}(\iota(F))$, where $\theta_{\R}$ is the $\R$-linear extension of $\theta$.
    \item Two abstract $\eta$-movable cones $F\subset \Lambda_0\otimes\R$ and $F'\subset \Lambda_0'\otimes\R$ are \emph{isomorphic} if there exists an isomorphism of primitive embeddings
     \[
      \xymatrix{
      \Lambda_0\ar@{->}[r]^-{\iota}\ar@{->}[d]^-{\bar{g}} & \Lambda_{\eta}\ar@{->}[d]^-{g} \\
      \Lambda_0'\ar@{->}[r]^-{\iota'} & \Lambda_{\eta}
      }
    \]
    such that the $\R$-linear extension of $g$ sends $F$ to $F'$.
    \end{enumerate}
    \end{defin}

    \begin{lem}\label{lem:finmanymovcones}
    Let $\eta$ be one of the known deformation classes of {\hk} manifolds. If $\Lambda_0$ is a lattice of signature $(1,n)$, then there are finitely many, up to isomorphism, abstract $\eta$-movable cones in $\Lambda_0\otimes \R$.
    \end{lem}
    \begin{proof}
    First of all, once $\eta$ is fixed, we have at most a finite number of choices for the set $\mathcal{S}_\eta$ of numerical stably exceptional classes (see Remark~\ref{rmk:S tau}). Moreover, by \cite[Remark 1.1]{Huyb18}, there are finitely many primitive embeddings of $\Lambda_0$ inside $\Lambda_{\eta}$, up to $\operatorname{O}(\Lambda_{\eta})$. 

    Let us then fix a primitive embedding $\iota\colon\Lambda_0\to\Lambda_{\eta}$ and a set $\mathcal{S}_\eta$ of numerical stably exceptional classes. 

    Let us choose a Hodge structure on $\Lambda_{\eta}$ such that $\Lambda_0 \otimes \R\subset (\Lambda_{\eta} \otimes\R)^{{1,1}}$. The cone of positive classes inside $(\Lambda_{\eta} \otimes\R)^{{1,1}}$ has two connected components; let us choose one of them and denote it by $\mathcal{C}_0$. Then there is a well-defined set $\mathcal{S}_0=\left\{x\in \mathcal{S}_\eta\cap (\Lambda_0\otimes \R) \mid (x,\alpha)>0 \text{ for some } \alpha \in\mathcal{C}_0\right\}$, which is independent on the Hodge structure chosen. Consider the wall-and-chamber decomposition 
    \begin{equation}\label{eqn:dec of C_0 into abstract mov} 
    \mathcal{C}_0 \setminus\bigcup_{x\in\mathcal{S}_0} x^\perp. 
    \end{equation}
    If $\mathcal{M}$ is a chamber, then by Lemma~\ref{lem:charMov} the subset $\mathcal{M}\cap(\Lambda_0 \otimes \R)$ is an abstract $\eta$-movable cone. 

    Finally, by acting with $\operatorname{O}(\Lambda_{\eta})$ (more precisely, with the $\R$-linear extensions of its elements), we obtain isomorphic abstract $\eta$-movable chambers, thus concluding the proof.
    \end{proof}

\begin{rem}
    


    If $\eta$ is one of the known deformation types, then there exists a finite-index subgroup
\[
M_\eta \subset \operatorname{O}(\Lambda_{\eta})
\]
with the property that, for every marked pair $(X,\theta)$ of class $\eta$,
\[
\Mon^2(X) = \theta^{-1} \circ M_\eta \circ \theta .
\]
This is established in \cite{Markman:Mon} for manifolds of type $\operatorname{K3}^{[n]}$; 
in \cite{Markman:MonJac,Mongardi:Mon} for manifolds of type $\operatorname{Kum}^n$; 
in \cite{MR:OG6} for manifolds of type OG6; 
and in \cite{Ono:Mon} for manifolds of type OG10.

Let $\operatorname{O}(\Lambda_{\eta}, \Lambda_0)$ denote the subgroup of isometries preserving $\Lambda_0$, and set
\[
M_0 = M_\eta \cap \operatorname{O}(\Lambda_{\eta}, \Lambda_0).
\]
At the end of the proof above, we can distinguish two cases. Since stably exceptional divisors are preserved by parallel transport operators, the action of $M_0$ does not change the set $\mathcal{S}_0$; hence the decomposition~(\ref{eqn:dec of C_0 into abstract mov}) remains invariant.

Acting on~(\ref{eqn:dec of C_0 into abstract mov}) with an isometry
\[
g \in \operatorname{O}(\Lambda_{\eta}, \Lambda_0) \setminus M_0
\]
produces a different (but isomorphic) decomposition. To make this explicit, first suppose that $\eta$ is a known deformation type for which
\[
M_\eta = \operatorname{O}^+(\Lambda_{\eta}),
\]
the group of orientation-preserving isometries. This holds when $\eta$ is of type $\operatorname{K3}^{[n]}$ with $n-1$ a prime power, or of type OG6 or OG10. In this case, the set $\mathcal{S}_\eta$ is numerically determined by the degree and divisibility of its elements, and the claim is clear.

In the remaining cases, the set $\mathcal{S}_\eta$ is determined by the degree and divisibility of its elements together with an additional numerical invariant depending on the choice of a primitive embedding of $\Lambda_{\eta}$ into the Mukai lattice $\widetilde{\Lambda}$ (see Remark~\ref{rmk:S tau}). The stabilizer of a primitive embedding $\iota \colon \Lambda_{\eta} \to \widetilde{\Lambda}$ (up to $\operatorname{O}(\widetilde{\Lambda})$) is the subgroup $\mathcal{W} \subset \operatorname{O}(\Lambda_{\eta})$ of isometries acting as $\pm \id$ on the discriminant group. If $\eta$ is of type $\operatorname{K3}^{[n]}$, then $M_\eta = \mathcal{W}$; if $\eta$ is of type $\operatorname{Kum}^n$, then $M_\eta \subset \mathcal{W}$ has index $2$ (cf.\ also \cite[Corollary~4.4]{Wieneck18}). Acting with
\[
g \in \operatorname{O}(\Lambda_{\eta}, \Lambda_0) \setminus M_0
\]
has the effect of changing the primitive embedding of $\Lambda_{\eta}$ into $\widetilde{\Lambda}$. This in turn modifies the set of numerically stably exceptional classes accordingly, that is,
\[
\xymatrix{
\mathcal{S}_\eta \subset \Lambda_{\eta} \ar[r]^-{\iota} \ar[d]^-{g} 
  & \widetilde{\Lambda} \ar[d]^-{\id} \\
g(\mathcal{S}_\eta) \subset \Lambda_{\eta} \ar[r]^-{\iota'} 
  & \widetilde{\Lambda}
}
\]
where $\iota' = \iota \circ g^{-1}$. Therefore, acting with
$g \in \operatorname{O}(\Lambda_{\eta}, \Lambda_0) \setminus M_0$
on~(\ref{eqn:dec of C_0 into abstract mov}) yields a decomposition into abstract $\eta$-movable cones that is, by construction, isomorphic to the original one.

If we do not require $\Lambda_0$ to be preserved, then we are simply changing $\Lambda_0$ within its $\operatorname{O}(\Lambda_{\eta})$-orbit of primitive embeddings in $\Lambda_{\eta}$, and we may therefore ignore this case without loss of generality.

\end{rem}

\section{Birational automorphism group of the very general fiber}
Let $\varphi\colon\mathscr{X}\to \Delta$ be a nontrivial family of \hk\ manifolds over the unit disk, with the Euclidean topology. Moreover, let $\tau \colon R^2\varphi_*\Z_\mathscr{X}\to \Lambda\times \Delta$ be a marking, where $\Lambda$ is the lattice given by the second integral cohomology group $H^2(\mathscr{X}_0, \Z)$, endowed with the Beauville--Bogomolov--Fujiki form.

Set $\Lambda_t \coloneqq \tau_t(N^1(\mathscr{X}_t))$. By \cite[Section 2]{Oguiso03}, there exists a primitive lattice $\Lambda_0\subset \Lambda$ and an uncountable dense subset $\cG\subset \Delta$ such that:
\begin{enumerate}
    \item $\cS\coloneqq \Delta\setminus \cG$ is countable and dense;
    \item $\Lambda_0\subset \Lambda_t$ for all $t\in \Delta$;
    \item $\Lambda_0 = \Lambda_t$ for all $t\in \cG$;
    \item $\Lambda_0\subsetneq \Lambda_t$ for all $t\in \cS$.
\end{enumerate}
We will say that the points in $\cG$ are \emph{very general}, while those in $\cS$ are \emph{special}. Note that the Picard number $\rho(\mathscr{X}_t)$ determines whether the point $t$ is very general or special.

For any fiber $\mathscr{X}_t$, we have a representation 
$$
\rho_t \colon \Bir(\mathscr{X}_t)\lra \text{O}^+(H^2(\mathscr{X}_t,\Z)).
$$
We denote by $\Mon^2_{\Bir}(\mathscr{X}_t)$ the image of $\rho_t$. Recall that these morphisms have finite kernel, for example by \cite[Proposition 9.1]{Huybrechts:Basic}.

Suppose that the family $\varphi \colon \mathscr{X}\to \Delta$ is projective, with a $\varphi$-ample line bundle $\cL$ on $\mathscr{X}$. Denote by $\ell$ the class $\tau_t([\cL_{\mathscr{X}_t}])$, which in independent of $t$. In particular, $\ell$ lies in $\Lambda_t$ for all $t$, hence we have $\ell\in \Lambda_0$. 


We set
\begin{equation}\label{eqn:defMov}
    M_t \coloneqq \tau_t(\Int(\Mov(\mathscr{X}_t)))\subset \Lambda_t, \quad\text{ and } M_t^0 \coloneqq M_t\cap \Lambda_0,
\end{equation}
for all $t\in\Delta$, where $\Mov(\mathscr{X}_t)$ is the movable cone of $\mathscr{X}_t$. 
Since $\cL$ is $\varphi$-ample, the class $\ell$ lies in $M_t^0$ for all $t\in \Delta$.

    \begin{lem}
        For any $t\in \cG$, the cone $M\coloneqq M_t\subset \Lambda_0$ is independent of $t$. Moreover, there is an inclusion
        $$
            M_t^0\subset M\quad \text{for all } t\in \Delta.
        $$
    \end{lem}
    \begin{proof}
        We fix a point $p \in \cG$. It suffices to show that, for any $t \in \Delta$, one has
\[
    M_t^0 \subset M_p.
\]
Indeed, for $t \in \cG$, this implies $M_t = M_t^0 \subset M_p$; exchanging the roles of $t$ and $p$ then yields the reverse inclusion, and hence equality.

Fix a class $h \in M_t^0$. Suppose, for the sake of contradiction, that $h$ is not movable on $\mathscr{X}_p$. Then, by Lemma~\ref{lem:charMov}, there exists a prime exceptional divisor $D$ on $\mathscr{X}_p$ such that
\[
    q_{\mathscr{X}_p}\big(\tau_p([D]), h\big) \le 0.
\]
Set $\alpha = \tau_p([D])$. 

Consider now the class $\tau_t^{-1}(\alpha) \in H^2(\mathscr{X}_t, \Z)$. This class is stably exceptional: indeed, $[D]$ is effective on $\mathscr{X}_p$ by definition, so $q_X(\alpha, \ell) > 0$ since $\ell$ is ample on $\mathscr{X}_p$. Moreover, the deformation $\varphi$ induces a parallel transport operator
\[
    T_\varphi^{t,p} \colon H^2(\mathscr{X}_t, \Z) \longrightarrow H^2(\mathscr{X}_p, \Z)
\]
such that $T_\varphi^{t,p}\big(\tau_t^{-1}(\alpha)\big) = \tau_p^{-1}(\alpha) = [E]$, where $E$ is the class of a prime exceptional divisor on $\mathscr{X}_p$.

Since $h$ is movable on $\mathscr{X}_t$, Lemma~\ref{lem:charMov} gives $q_{\mathscr{X}_t}(\alpha, h) > 0$, contradicting the previous inequality. This proves the claim.
\end{proof}

Let $\mathscr{X}_t$ be any fiber of the family $\varphi$. Recall that the subgroup $\Mon^2(\mathscr{X}_t)$ of monodromy operators in $\mathrm{O}\big(H^2(\mathscr{X}_t, \Z)\big)$ is a deformation invariant. Thus, via the marking $\tau$, we may identify it with a fixed subgroup
\[
    \Mon^2(\Lambda) \subset \mathrm{O}(\Lambda).
\]

Let $\widetilde{\mathrm{O}}(\Lambda_0)$ denote the subgroup of isometries of $\Lambda_0$ acting trivially on its discriminant group $D_{\Lambda_0}$.
Since $\Lambda_0$ is a primitive sublattice of $\Lambda$, any isometry
$\widetilde{g}_0 \in \widetilde{\mathrm{O}}(\Lambda_0)$ extends uniquely to an isometry
$g \in \mathrm{O}(\Lambda)$ such that
\[
    g|_{(\Lambda_0)^{\perp}} = \mathrm{id}
\]
(see, for instance, \cite[Proposition~2.6]{Huyb16}).
This defines an embedding
\[
    j \colon \widetilde{\mathrm{O}}(\Lambda_0) \hookrightarrow \mathrm{O}(\Lambda),
\]
sending $\widetilde{g}_0$ to its extension $g$.
We then set
\begin{equation}\label{eq:defMonLambda0}
    \Mon_{\Lambda_0}
    \coloneqq
    \Mon^2(\Lambda) \cap j\big(\widetilde{\mathrm{O}}(\Lambda_0)\big)
    \subset \Mon^2(\Lambda).
\end{equation}

For any $t \in \Delta$, every $g \in \Mon_{\Lambda_0}$ satisfies
$\tau_t^*(g) \in \Mon^2_{\Hdg}(\mathscr{X}_t)$.
Indeed, since $\tau_t(\Lambda_0) \subset N^1(\mathscr{X}_t)$, the transcendental lattice of $\mathscr{X}_t$
is contained in $\tau_t(\Lambda_0)^{\perp}$; hence $\tau_t^*(g)$ fixes the symplectic form of $\mathscr{X}_t$,
and thus lies in $\Mon^2_{\Hdg}(\mathscr{X}_t)$.

For each $t$, we define
\begin{equation}\label{eq:defGt}
    G_t
    \coloneqq
    \bigl\{\, g \in \Mon_{\Lambda_0} \;\big|\; g(\ell) \in M_t \,\bigr\}.
\end{equation}
Since $M_t = M$ for all $t \in \cG$, the group
\[
    G \coloneqq G_t
\]
is independent of the choice of $t \in \cG$.

\begin{lem}\label{lem:groupG_t}
    With the notation above, the marking $\tau_t^*$ induces an inclusion $G_t \hookrightarrow \Mon^2_{\Bir}(\mathscr{X}_t)$ with image $\Mon^2_{\Bir}(\mathscr{X}_t)\cap \tau_{t}^*(\Mon_{\Lambda_0})$.
    Moreover, for any $t\in \cG$ the group $G = G_t$ is finitely generated and there is a constant $N\coloneqq N(\Lambda_0)$ such that
    $$
    [\Mon^2_{\Bir}(\mathscr{X}_t) : \tau^*_t(G)] \le N.
    $$
\end{lem}
\begin{proof}
We start by claiming that \begin{equation}\label{eqn:def2MonBir}
    \Mon^2_{\Bir}(\mathscr{X}_t)
    =
    \left\{
        g \in \Mon^2_{\Hdg}(\mathscr{X}_t)
        \,\middle|\,
        g(\ell) \in \Int\big(\Mov(\mathscr{X}_t)\big)
    \right\}.
\end{equation}
The inclusion $\subseteq$ is trivial, so we prove $\supseteq$. We call a \emph{wall} of the nef cone of a projective {\hk} manifold a top-dimensional extremal face. According to \cite[Definition~1.2, Lemma~1.4]{Mongardi15}, monodromy operators in $\Mon^2_{\text{Hdg}}(\mathscr{X}_t)$ send walls of $\mathrm{Nef}(\mathscr{X}_t)$ to walls of the nef cone of a birational model of $\mathscr{X}_t$, and vice versa. Since $\ell$ is ample, it follows that there exists a birational model of $X$ on which the class $g(\ell)$ becomes ample (possibly non-isomorphic to $\mathscr{X}_t$).

Indeed, since $g(\ell)$ lies in the interior of the movable cone, it can fail to be ample only if it lies on a wall of the nef cone of some birational model of $\mathscr{X}_t$. However, by acting with $g^{-1}$ and using the $\Mon^2_{\Hdg}$-invariance of walls, we would then obtain that $\ell$ lies on a wall of the nef cone, which is a contradiction.

To conclude, arguing as in \cite[Proof of Corollary~5.7]{Markman11}, we deduce that any $g$ as in~\eqref{eqn:def2MonBir} is induced by a birational automorphism, that is, $g \in \Mon^2_{\text{Bir}}(X)$. 

Now, since $\tau_t^*(\Mon_{\Lambda_0}) \subset \Mon^2_{\Hdg}(\mathscr{X}_t)$ and
$\tau_t\big(\Int(\Mov(\mathscr{X}_t))\big) = M_t$, it follows that
\[
    \tau_t^* G_t = \Mon^2_{\Bir}(\mathscr{X}_t) \cap \Mon_{\Lambda_0}.
\]

Let now $t$ be any point in $\cG$.  
By \cite[Proof of Theorem~2]{BS12}, the group
\[
    \Gamma_{T(\mathscr{X}_t), \Bir}
    \coloneqq
    \left\{
        g \in \Mon^2_{\Bir}(\mathscr{X}_t)
        \,\middle|\,
        g_{|T(\mathscr{X}_t)} = \id
    \right\}
\]
is finitely generated (note that our notation differs slightly from that used in \cite[Proof of Theorem~2]{BS12}). Moreover, for every $t \in \cG$ we have a natural group homomorphism
\[
    r \colon \Gamma_{T(\mathscr{X}_t), \Bir} \to \text{Hom}(D(\Lambda_0),D(\Lambda_0)),
\]
since $\Lambda_0 \cong N^1(\mathscr{X}_t)$ for all very general $t$.
By construction, $\ker(r) = G$.

It follows that $G$ has finite index in
$\Gamma_{T(\mathscr{X}_t), \Bir}$, and hence $G$ is finitely generated.
The second statement then follows from the fact that
$\Gamma_{T(\mathscr{X}_t), \Bir}$ has finite index in $\Mon^2_{\Bir}(\mathscr{X}_t)$
(see \cite[Lemma~2]{BS12}),
and that $G$ has finite index in
$\Gamma_{T(\mathscr{X}_t), \Bir}$.
\end{proof}

We are now ready to prove Theorem \autoref{thm:main1}.

\begin{proof:mainthm1}
By Lemma \autoref{lem:groupG_t} the group $G:= G_t$ is well-defined when $t\in \mathcal{G}$. Let $t_0$ be any point of $\mathcal{G}$. We set
\[
\mathcal{F}:=\{t\in \mathcal{S}\mid G_{t_0}\not \subset G_t\}.
\]
We claim that $\mathcal{F}$ is a finite set. By Lemma \autoref{lem:groupG_t}, the group $G_{t_0}$ is a finitely generated subgroup of $\text{O}(\Lambda_0)$, so we can choose a finite set of generators $\{g_i\}_{i=1}^k$ for $G_{t_0}$. To prove the claim it suffices then to show that the sets
\[
\mathcal{F}_i:=\{t\in \mathcal{S}\mid g_i \not \in G_t\}
\]
are finite. Choose $g=g_j\in \{g_i\}_{i=1}^k$, and let $t$ be an element of $\mathcal{F}_j$. 
Then $g(\ell)\notin M_t$, by definition of $\mathcal{F}_j$. 
Note that $g(\ell)\in \Lambda_0$, since $g\in \Mon_{\Lambda_0}$. 
Moreover, $q(g(\ell))>0$ and $q(g(\ell),\ell)>0$, because $\ell,g(\ell)\in M_{t_0}$. 
In particular, $g(\ell)$ lies in $\mathrm{Pos}(\mathscr{X}_t)$, and hence $g(\ell)$ is represented 
by an effective $\Q$-divisor $D_t$ (in fact, we may take $D_t$ to be integral, 
by \cite[Corollary~5.3]{Jiang23}).

Since $g(\ell)\notin M_t$, Lemma~\autoref{lem:charMov} implies that there exists a 
prime exceptional divisor $E_t$ on $\mathscr{X}_t$ such that
\[
    q\big(g(\ell),[E_t]\big)\le 0.
\]
The extension $\widetilde{g}$ of $g$ to $\Lambda$, characterized by 
$\widetilde{g}_{|(\Lambda_0)^{\perp}}=\id$, is a Hodge isometry.  
Hence either $\widetilde{g}^{-1}([E_t])=E'_t$ or 
$-\widetilde{g}^{-1}([E_t])=-E'_t$ is effective.

Since $\ell$ is ample, we in fact have
\[
    q\big(\ell,[E'_t]\big)=q\big(g(\ell),[E_t]\big)\neq 0,
\]
and therefore $E_t$ must appear in the support of $D_t$.  
In particular, for every $t\in \mathcal{F}_i$ we find a prime exceptional divisor $E_t$ such that
\[
    0 < q\big(\ell,[E_t]\big)\leq q\big(\ell,g(\ell)\big),
\]
because $E_t$ is contained in the support of $D_t$.  
Note that the number $q(\ell,g(\ell))$ is independent of~$t$.

Consider now the relative Hilbert scheme $\mathrm{Hilb}^\ell_{\mathscr{X}/\Delta}$ for the relative polarization $\ell$, 
which exists for projective analytic families (see, for example, 
\cite[Chapter~IX, Section~7]{ACGH11}).  
If there are infinitely many points $t\in \mathcal{F}_i$, then we obtain 
infinitely (in fact, countably) many prime exceptional divisors $\{E_t\}_t$ on 
fibers $\mathscr{X}_t$ satisfying
\[
    0 < q\big(\ell,[E_t]\big)\leq q\big(\ell,g(\ell)\big).
\]

The Hilbert polynomial 
\[
    P_{E_t}(m)=\chi\big(\mathscr{O}_{E_t}(m\mathcal{L}_t)\big)
\]
is determined by the quantities 
$q(\mathcal{L}_t,E_t)$, $q(E_t)$, and $q(\mathcal{L}_t)$.  
Since $q(E_t)$ is bounded from below (\cite[Proposition 4.8]{KMPP19}) and $q(\mathcal{L}_t,E_t)$ is bounded from above, 
only finitely many distinct Hilbert polynomials can occur.  
Hence, one component of the relative Hilbert scheme must contain infinitely many of the 
divisors $E_t$; denote this component by $\mathscr{H}$.  

The morphism $\mathscr{H}\to \Delta$ is dominant and, being proper, has to be surjective, by Remmert's proper mapping Theorem.  
Consequently, on the fiber $\mathscr{X}_{t_0}$ there exists an effective 
divisor $E_{t_0}$ such that
\[
    q\big(E_{t_0},g(\ell)\big)<0,
\]
which is impossible, because $g(\ell)$ lies in the interior of the movable cone of $\mathscr{X}_{t_0}$.
\end{proof:mainthm1}

\section{Birational automorphism group of the special fiber}

We follow the notation of the previous section and start this one with the following auxiliary lemma.

\begin{lem}\label{lem:schifo}
    Let $M$ be a lattice and consider a rank two sublattice $L = \Z\langle \ell, a\rangle$, where $\ell$ is a primitive vector in $M$ of positive square. Suppose that $M$ contains a copy of $U$ which is orthogonal to $L$. Then, for any positive integer $N$, there exists a vector $h\in M$ and infinitely many values of $n$ such that the lattice $\Z\langle \ell, na + h\rangle$ is primitive in $M$ and does not contain vectors of square $-i$ for all $0\le i\le N$. 
\end{lem}

\begin{proof}
    Let  $A = \ell^2$ be the square of $\ell$. For any integer $d$ we define
    $$
        Q(d) \coloneqq \prod_{p\mid d} p^{2\lceil \frac{v_p(d)}2\rceil},
    $$
    where the product is taken over the prime factors of $d$, and $v_p(d)$ is the $p$-adic valuation of $d$. We have $d\mid Q(d)$, and $Q(d)\mid d$ if and only if $d$ is a square. Moreover, for any integer $x$ such that $d\mid x^2$, we have $Q(d)\mid x^2$. Finally, we set 
    $$
    m_1 \coloneqq 4\prod_{0<Ad\le N} Q(d).
    $$

    We go back to the proof, and choose a primitive vector $h\in U$ of square $h^2=Am$, for some non-square integer $m$ which is divisible by $m_1$. Then $h$ is orthogonal to $\ell$ and $a$, and the lattice 
    $$
        L_{n, h} \coloneqq \Z\langle \ell , (mA^2n)a + h\rangle,
    $$
    is primitive in $M$. Indeed, since $U$ is unimodular, there exists $k\in U$ such that $k\cdot h=1$. Then, any vector $v\in  M$ such that $rv$ lies in $L_{n, h}$ for some nonzero integer $r$, can written as $\frac{x}{r}\ell + \frac{y}{r}((mA^2n)a + h)$. Therefore $\frac{y}{r} = v\cdot k$ is an integer, and hence the vector $\frac{x}{r}\ell$ lies in $M$. Since $\ell$ is primitive, $\frac{x}{r}$ is also an integer and $v$ lies in $L_{n,h}$.
    

   We denote by $b$ the vector $Aa - (\ell\cdot a)\ell$. Then the vector $b$ is orthogonal to $\ell$ and to $h$, and we have the equality of lattices
   $$
   L_{n,h} =\Z\langle \ell , (mA^2n)a + h\rangle = \Z\langle \ell, (mAn)b + h\rangle.
   $$
   In particular, any vector $v\in L_{n,h}$ can be written as $x\ell + y((mAn)b + h)$ and has square 
   \begin{equation*}
       v^2 = Ax^2 + Am(mAn^2\cdot (b^2) + 1)y^2.
   \end{equation*}
   We show that the equation $v^2=-i$ has no solution for $0\le i \le N$. Since $A\mid v^2$, for any $v\in L_h$, we just consider integers $i\le N$ of the form $i = Ad$ 
   , and show that the equation 
    \begin{equation}\label{eqn:square}
       x^2 + m(mAn^2\cdot (b^2) + 1)y^2 = -d
   \end{equation}  
   has no solution.

   For $d=0$, just notice that $|m(mAn^2\cdot (b^2) + 1)|$ is not a square, since $m$ is not a square by hypothesis and is coprime with $(mAn^2\cdot (b^2) + 1)$, therefore $|m(mAn^2\cdot (b^2) + 1)|y^2 \neq x^2$ for all integers $x,y$.\\

   Otherwise, $d$ must divide $x^2$, hence $Q(d)$ divides $x^2$ and, by definition of $m$, $Q(d)$ divides the left hand side of Equation~\eqref{eqn:square}, thus $d$ is equal to a square $q^2$. But then $q$ divides $x$ and $m$ is divisible by $q^2$ by definition. Hence we can divide Equation~\eqref{eqn:square} by $q^2$, and obtain
   $$
    x_1^2 + \frac{m}{q^2}(mAn^2\cdot (b^2) + 1)y^2 = -1,
   $$
   which has no solutions, since $4\mid \frac{m}{q^2}$ by definition, and $-1$ is not a square modulo $4$.
\end{proof}

We are now ready to prove Theorem \ref{thm:main2}.

\begin{proof:mainthm2}
    We show that, for any point $t_0\in \cG$ and any open neighborhood $U\subset \Delta$ of $t_0$, there exists a point $t\in \cS\cap U$ such that $|\Bir(\cX_t)|=\infty$.  Up to shrinking $\Delta$, we may assume that $\mathcal{F}=\emptyset$.
    
    Set $r_0 \coloneqq \rank(\Lambda_0)$ and recall that for any $t\in \cS$, we have $\rank(\Lambda_t)>r_0$. Therefore, up to shrinking $\Delta$, we can suppose $U=\Delta$ and that one of the two following cases apply:
    \begin{enumerate}[label = (\arabic*)]
        \item\label{rank3and2} there exists a sequence $\{t_n\}_{n\ge 1}\subset \cS$ converging to $p$ such that $\rank(\Lambda_{t_n})\ge 3$ for all $n$, and $\rank(\Lambda_0)\geq 2$;
        \item\label{rank3and1} there exists a sequence $\{t_n\}_{n\ge 1}\subset \cS$ converging to $p$ such that $\rank(\Lambda_{t_n})\ge 3$ for all $n$, and $\rank(\Lambda_0)=1$;
        \item\label{rank2} $r_0=1$ and any $t\in \cS$ satisfy $\rank \Lambda_t =2$.
    \end{enumerate}
    Then the theorem is reduced to showing that in the three cases there exists \emph{one} point $t\in \cS$ such that $|\Bir(\cX_t)|=\infty$.
    
    We first consider case~\ref{rank3and2}. We argue by contradiction and assume that $\mathrm{Bir}(\mathscr{X}_{t_n})$ is a finite group, for every $t \in \Delta$. Since $\mathrm{rk}(\Lambda_{t_n})\geq 3$ and $|\mathrm{Bir}(\mathscr{X}_{t_n})|<\infty$, it follows from the birational version of the Kawamata-Morrison cone conjecture (\cite[Theorem 6.25]{Markman11}) that $\mathscr{X}_{t_n}$ carries finitely many prime exceptional divisors only, say,
    \[
    E_{n,1},\dots , E_{n,k(n)},
    \]
    where $k(n)$ depends on the fiber $\mathscr{X}_{t_n}$. Moreover, by \cite[Theorem 1.2]{Den22}, we have that 
    \[
    \mathrm{Eff}(\mathscr{X}_{t_n})=\sum_{j=1}^{k(n)}\R_{\geq 0}[E_{n,j}].
    \]
    We now recall that $|\mathrm{Bir}(\mathscr{X}_{t_n})|<\infty$ and $\rho(\mathscr{X}_{t_n})\geq 3$, for every $n$. Thus, by Proposition \ref{prop:finmanylattices}, there are only finitely many lattices representing the Picard groups $\mathrm{Pic}(\mathscr{X}_{t_n})$. We may assume that $\rho(\mathscr{X}_{t_n})=:\rho$. By Lemma \ref{lem:finmanymovcones}, we may assume that any fiber $\mathscr{X}_{t_n}$ has a fixed number of prime exceptional divisors $k$, which is independent of $n$. Moreover, thanks to Lemma \ref{lem:finmanymovcones}, we may also assume that the intersection matrices $\left(q(E_{n,i},E_{n,j})\right)_{i,j}$ are constant.

    We claim that if the set $\{q_{\mathscr{X}_{t_n}}(\ell,E_{n,i})\}_n$ is bounded for every $i$, then we reach a contradiction. Indeed, this implies, as in the proof of Theorem~\ref{thm:main1}, that one component $\mathscr{H}$ of the relative Hilbert scheme $\mathrm{Hilb}^{\ell}_{\mathscr{X}/{\Delta}}$ for the polarization $l$ dominates $\Delta$. By the properness of the morphism $\mathscr{H}\to \Delta$, it follows that the divisors $E_{n,0},\dots E_{n,k}$ can be degenerated to effective divisors $E'_0,\dots,E'_k$ on the fiber $\mathscr{X}_{t_0}$, such that their Gram-matrix for the BBF form is the same as that of $E_{n,0},\dots E_{n,k}$. But then $\mathscr{X}_{t_0}$ has at least $\rho$ linearly independent effective divisor classes, hence $\rho(\mathscr{X}_{t_0})\geq \rho(\mathscr{X}_{t_n})>r_0$, a contradiction.
We are then reduced to showing that the set $\{q_{\mathscr{X}_{t_n}}(\ell,E_{n,i})\}_n$ is bounded for every $i$. We distinguish three subcases.
  
    \underline{Suppose that $\mathrm{rk}(\Lambda_{t_0})\ge 3$.} In this case, $\mathscr{X}_{t_0}$ has finitely many prime exceptional divisors generating the effective cone, say
    \[E_{0,1},\dots, E_{0,k(0)}.\]
    Now, being a prime exceptional divisor is an open property in families, by \cite[Proposition 5.2]{Markman13}, therefore, as the classes $[E_i]$ stay algebraic in the whole family, they are represented by prime exceptional divisors in a small neighborhood of $t_0$. In particular, it follows that they are represented by prime exceptional divisors on the fibers $\mathscr{X}_{t_n}$ for every sufficiently large $n$. For the fiber $\mathscr{X}_{t_n}$, we denote them by $\{E^n_{0,i},\dots,E^n_{0,k(0)}\}$. In particular
    \[
    \{E^n_{0,i},\dots,E^n_{0,k(0)}\}\subset \{E_{n,1},\dots, E_{n,k}\},
    \]
    and we can write $l=\sum_{j=0}^{k(0)}b_jE^n_{0,j}$ for all $n\gg 0$. Since the matrix $q_{\cX_{t_n}}(E_{n,i}, E_{n, j})$ is bounded, the set $\{q_{\mathscr{X}_{t_n}}(E_{n,i},l)\}_{n\geq 1}$ is bounded.
    
    \underline{Suppose that $\mathrm{rk}(\Lambda_{t_0})=2$, and $\mathrm{rk}(\Lambda_{t_n})\geq 3$, for any $n$.} By \cite[Corollary 1.6]{Den22}, \cite[item (2) of Theorem 1.3]{Oguiso14}, the two rays generating $\mathrm{Eff}(\mathscr{X}_{p})$ are rational. There are three possibilities: (1) the rays are spanned by isotropic classes of effective divisors $H_1$ and $H_2$, or (2) by an isotropic class of an effective divisor $H$ and a class of a prime exceptional divisor $E_{0,1}$, or (3) by classes of prime exceptional divisors $E_{0,1}$ and $E_{0,2}$. In case (3) the proof goes exactly as in the $\mathrm{rk}(\Lambda_{t_0})=3$ case, so it remains to consider (1) and (2).
    First, assume that $\mathrm{Eff}(\mathscr{X}_{p})$ is spanned by $[H_1]$ and $[H_2]$. By \cite[Theorem 1.2]{Matsu17}, the divisors $[H_1]$ and $[H_2]$ stay effective and movable on the nearby fibers of the family, and hence on the fibers $\mathscr{X}_{t_n}$ for $n\gg 0$. Moreover, we may assume that $[H_1]$ and $[H_2]$ are primitive classes on $\mathscr{X}_{t_0}$ and on the nearby fibers. Then we have $[H_j]=\sum_{i=1}^{k}a_{j,i}(t_n)[E_{n,i}]$ for $n\gg 0$, where, a priori, the coefficients $a_{j,i}(t_n)\geq0$ depend on $n$. We now recall that the Néron--Severi lattice of a projective hyper-Kähler manifold with finite birational automorphism group (and hence with rational polyhedral effective cone) possesses only finitely many primitive isotropic classes represented by movable divisors. But since the Néron--Severi lattice of the fibers $\mathscr{X}_{t_n}$ is fixed, as well as the movable cone and the intersection matrix of the prime exceptional divisors, then there are only finitely many possibilities for $a_{j,i}(n)$, for each $n$. Hence we can assume (up to passing to a subsequence) that the coefficients $a_{j,i}$ in the linear combinations $[H_j]=\sum_{i=1}^{k}a_{j,i}(t_n)[E_{n,i}]$ do not depend on $n$. Since $\ell$ is a linear combination of $[H_1]$ and $[H_2]$, we deduce that $\{q_{\mathscr{X}_{t_n}}(\ell,E_{n,i})\}_n$ is bounded for every $i$. The case where the effective cone of $\mathscr{X}_{t_0}$ is spanned by a prime exceptional divisor and by an isotropic integral class is treated in the same way. This concludes the proof of this subcase and hence of case~\ref{rank3and2}.

   We now focus on case~\ref{rank3and1}.

    \underline{Suppose that $\mathrm{rk}(\Lambda_0)=1$ and $\mathrm{rk}(\Lambda_{t_n})\ge 3$ for all $n$.} Choose a Euclidean norm $\lVert \cdot \rVert$ on $\Lambda \otimes \mathbf{R}$. We first observe that if $\{\lVert E_{n,i}\rVert\}_{n\ge 1}$ is bounded for some $i$, then the set $\{E_{n,i}\}_{n\ge 1}$ is finite, because there are only finitely many lattice points with bounded norm. It follows that the set $\{q_{\mathscr{X}_{t_n}}(E_{n,i},\ell)\}_{n\ge 1}$ is bounded, and then the sequence $\{E_{n,i}\}_n$
 gives an effective divisor of negative BBF square on the fiber $\mathscr{X}_{t_0}$, which has Picard rank $1$, and hence a contradiction.
    
    We may then assume that $\{\lVert E_{n,i}\rVert\}_{n\ge 1}$ is unbounded for all $i$. By passing to a subsequence, we may assume that $\lim_{n \to \infty} \lVert E_{n,i}\rVert = \infty$. Let $x_{n,i}$ be the image of $[E_{n,i}]/\lVert E_{n,i}\rVert$ under the marking 
    \[
    \tau \colon R^2 \varphi_* \mathbf{Z}_\mathscr{X} \to \Lambda \times \Delta.
    \]
    Since $\lVert x_{n,i}\rVert = 1$, we can pass to a convergent subsequence in $\Lambda \otimes \mathbf{R}$ and define
    \[
    x_{0,i} := \lim_{n\to\infty} x_{n,i} \in \Lambda \otimes \mathbf{R}.
    \]
    Since the set $\{q_{\mathscr{X}_{t_n}}(E_{n,i}, E_{n,j})\}_n \subset \mathbf{Z}$ is bounded, we have $q_{\mathscr{X}_{t_0}}(x_{0,i}, x_{0,j}) = 0$ for all $i,j$. Let $\omega$ be a nowhere vanishing section of $\varphi_* \Omega^2_{\mathscr{X}/\Delta}$. With some abuse of notation, we also denote by $x_{0,i}$ the corresponding class on the fiber $\mathscr{X}_{t_0}$ under $\tau$. We observe that
    \[
    q_{\mathscr{X}_{t_0}}(x_{0,i}, \omega_{\mathscr{X}_{t_0}}) = \lim_{n \to \infty} q_{\mathscr{X}_{t_n}}(x_{n,i}, \omega_{\mathscr{X}_{t_n}}) = 0,
    \]
    hence the $x_{0,i}$ are represented by real $(1,1)$-forms on $\mathscr{X}_{t_0}$. Since the signature of $q_{\mathscr{X}_{t_0}}$ restricted to $H^{1,1}(\mathscr{X}_{t_0}, \mathbf{R})$ is $(1, b_2(\mathscr{X}_{t_0})-3)$, the dimension of a totally isotropic subspace in $H^{1,1}(\mathscr{X}_{t_0}, \mathbf{R})$ is at most one. Therefore, there exists $e \in H^{1,1}(\mathscr{X}_{t_0}, \mathbf{R})$ with $q_{\mathscr{X}_{t_0}}(e,e) = 0$ such that $x_{0,i} \in \mathbf{R} e$ for all $i$. Consequently, the rational polyhedral cones $\mathrm{Eff}(\mathscr{X}_{t_n})$ accumulate toward the ray $\mathbf{R} e$ as $n \to \infty$. Since $\ell \in \mathrm{Eff}(\mathscr{X}_{t_0})$, we obtain $\ell \in \mathbf{R} e$. However, $q_{\mathscr{X}_{t_0}}(\ell,\ell) > 0$, which is a contradiction.

    
    To conclude, we consider case~\ref{rank2}. 
    
    \underline{Suppose that $\mathrm{rk}(\Lambda_0)=1$ and $\mathrm{rk}(\Lambda_{t_n})= 2$ for all $n$.} Let $$\pi \colon\Delta \to \cP = \{x\in \P(\Lambda\otimes \C) \mid
 x^2 = 0, \quad x\cdot \bar{x} >0\},$$
    be the period map associated with the family $\cX\to \Delta$, which sends each point $t\in \Delta$ to $[\tau_t(\sigma_t)]$, where $\sigma_t$ is the symplectic form on $\cX_t$. For any $v\in \Lambda$, we denote by $H_v$ the divisor cut out by $\P(v^\perp \otimes \C)\subset \P(\Lambda\otimes \C)$. The point $\pi(t)$ lies in $H_v$ if and only if the vector $\tau_t^{-1}(v)$ lies in $N^1(\cX_t)$.
    
    By assumption, the lattice $\Lambda_{t_0}$ has rank $1$, hence it is generated by the primitive polarization $\ell$ and the image $\pi(\Delta)$ is contained in $H_{\ell}$. Moreover, since for any $t\in \cS$, the lattice $\Lambda_t$ has rank $2$, a point $t\in \Delta$ lies in $\cS$ if and only if there exists $a\in \Lambda\setminus \Z_{\ell}$ such that $\pi(t)\in H_a$. 

    Let $t$ be a point in $\cS$. If the group $\Bir(\cX_t)$ is infinite, then we are done. Otherwise, let $\ell$ and $a$ be generators of the lattice $\Lambda_t$. 

    We now recall that, for any \hk \ manifold $X$ of Picard rank $2$, the group $\mathrm{Bir}(X)$ is finite if and only if $N^1(X)$ does not contain neither isotropic classes, nor classes representing prime exceptional divisors (see, for example, \cite[Corollary 1.6]{Den22}). Moreover, for any fixed deformation type, the square of prime exceptional divisors is bounded from below (\cite[Proposition 4.8]{KMPP19}).
    
    Lemma~\ref{lem:schifo} shows that there exists a vector $h\in \Lambda$ and infinitely many values of $n$ such that the lattice $L_{n,h} =\Z\langle \ell , na + h\rangle$ is primitive in $\Lambda$ and does not contain any vector of square equal to $-i$ for all $0\le i\le N$.

    Since the $H_a\cap \pi(\Delta)$ is nonempty, there exists $n$ large enough such that $\pi(\Delta)\cap H_{a + \frac{h}n}$ is nonempty, and in particular a point $\bar t\in \cS$ with $N_{\bar t}\subset L_{n,h}$. By the discussion above, the \hk \ manifold $\cX_t$ satisfies $|\Bir(\cX_t)|=\infty$.
\end{proof:mainthm2}

\begin{rem}
Let $\mathscr{X}\to \Delta$ be a nontrivial projective family of {\hk} manifolds belonging to one of the known deformation classes. Suppose that $\mathrm{rk}(\Lambda_0)=1$ and $\mathrm{rk}(\Lambda_t)=2$ for infinitely many $t \in \mathcal{S}$. It is worth noting that the proof of the last case above shows that the set of points in $\Delta$ where the automorphism group is infinite is dense in $\Delta$. Indeed, a projective {\hk} manifold $X$ with $b_2(X)\ge 5$ and Picard rank equal to $2$ has infinite $\mathrm{Aut}(X)$ if its Picard lattice does not represent nef integral isotropic classes, nor certain classes $\alpha$ of negative square such that $|q_X(\alpha)|\le N$, for some fixed positive integer $N$ depending only on the deformation class of the {\hk} manifolds.

This can be seen as follows: by \cite[Theorem 1.3]{Oguiso14}, the nef cone of a {\hk} manifold as above is rational polyhedral if and only if $\mathrm{Aut}(X)$ is finite. In this case, the extremal rays are spanned either by isotropic integral classes or by primitive divisor classes orthogonal to the so-called \enquote{wall-divisor} classes, equivalently known as \enquote{MBM classes}. We refer the reader to \cite{Mongardi15} and the survey \cite{AV21} for more information on wall-divisors, MBM classes, and their equivalence. The key point here is that the BBF square of primitive MBM classes is bounded, with the bound depending only on the deformation class, thanks to \cite[Theorem 1.5, Corollary 1.6]{AV17} and \cite[Theorem 3.17]{AV20}. Then, once we know that the Picard lattice represents neither negative classes $\alpha$ with $|q_X(\alpha)|\le N$ (where $-N$ is the bound on the BBF square of MBM classes) nor isotropic classes, we can conclude that the automorphism group is infinite.
\end{rem}

We conclude the paper by showing that, for a family over a disk, the members with infinite birational automorphism group form a dense subset of the base even when the family is not projective.

\begin{proof:corollary}
    We argue by contradiction. Suppose that the set of points $\mathcal{D}$ where the birational automorphism group is infinite is not dense. Consider then $\Delta^{\text{o}}=\Delta \setminus \overline{\mathcal{D}}$ and the associated family $\varphi_{|\mathscr{X}_{\Delta^{\text{o}}}}\colon \mathscr{X}_{\Delta^{\text{o}}}\to \Delta^{\text{o}}$. We may assume that $\Delta^{\text{o}}$ is connected. Since $\mathcal{G}$ is dense in $\Delta$, there exists a point $p\in\mathcal{G}\cap \Delta^{\text{o}}$. Moreover, the restriction $\varphi_{|\mathscr{X}_{\Delta^{\text{o}}}}$ is proper, and $\Lambda_0 \subset \Lambda_t$, for any $t$. Thus $\varphi_{|\mathscr{X}_{\Delta^{\text{o}}}}$ is a projective morphism around $p$. Let $U=U(p)$ be a small open neighborhood where $\varphi_{|\mathscr{X}_{\Delta^{\text{o}}}}$ is projective.  By the density of $\mathcal{S}\cap \Delta^{\text{o}}$ in $\Delta^{\text{o}}$, the restriction of $\varphi_{|\mathscr{X}_{\Delta^{\text{o}}}}$ to $U$ stays nontrivial (in particular, there are fibers where the Picard rank jumps). But then by Theorem \ref{thm:main2} there are points in $U$ where the birational automorphism group is infinite, which is a contradiction. Then $\mathcal{D}$ is dense in $\Delta$, and we are done.
\end{proof:corollary}

\bibliography{biblio}
\bibliographystyle{alpha}

\end{document}